\documentclass[12pt]{amsart}

\usepackage{etex}
\usepackage{amsmath, amssymb}
\usepackage{array}
\usepackage[frame,cmtip,arrow,matrix,line,graph,curve]{xy}
\usepackage{graphpap, color, paralist, pstricks}
\usepackage[mathscr]{eucal}
\usepackage[pdftex]{graphicx}
\usepackage[pdftex,colorlinks,backref=page,citecolor=blue]{hyperref}
\usepackage{tikz}
\usepackage{tikz-cd}

\newtheorem{theorem}{Theorem}[section]
\newtheorem{theoremletter}{Theorem}

\newtheorem{proposition}[theorem]{Proposition}
\newtheorem{corollary}[theorem]{Corollary}
\newtheorem{lemma}[theorem]{Lemma}

\newtheorem{conjecture}[theorem]{Conjecture}

\theoremstyle{definition}
\newtheorem{definition}[theorem]{Definition}
\newtheorem{example}[theorem]{Example}

\newtheorem{question}[theorem]{Question}
\newtheorem{remark}[theorem]{Remark}

\newcommand{\CC}{\mathbb{C} }
\newcommand{\PP}{\mathbb{P} }
\newcommand{\QQ}{\mathbb{Q} }
\newcommand{\RR}{\mathbb{R} }

\newcommand{\ZZ}{\mathbb{Z} }

\newcommand{\cE}{\mathcal{E} }

\newcommand{\cM}{\mathcal{M} }

\newcommand{\cO}{\mathcal{O} }

\newcommand{\cV}{\mathcal{V} }

\newcommand{\rD}{\mathrm{D} }
\newcommand{\rH}{\mathrm{H} }
\newcommand{\rM}{\mathrm{M} }
\newcommand{\rN}{\mathrm{N} }

\newcommand{\ba}{\mathbf{a} }

\newcommand{\bp}{\mathbf{p} }

\newcommand{\bm}{\mathbf{m} }
\newcommand{\bn}{\mathbf{n} }
\newcommand{\bM}{\mathbf{M} }

\def\CParHom{\mathcal{P}ar\mathcal{H}om }

\def\SL{\mathrm{SL}}

\def\Gr{\mathrm{Gr}}

\def\Ext{\mathrm{Ext} }
\def\git{/\!/ }
\def\Pic{\mathrm{Pic} }

\def\rk{\mathrm{rank}\, }
\def\im{\mathrm{im}\, }

\def\proj{\mathrm{Proj}\;}
\def\Eff{\mathrm{Eff}}
\def\pdeg{\mathrm{pardeg}}

\def\Det{\mathrm{Det}}

\begin{document}

\title[Positivity of the Poincar\'e bundle and applications]{Positivity of the Poincar\'e bundle on the moduli space of vector bundles and its applications}
\date{\today}

\author{Kyoung-Seog Lee}
\address{Kyoung-Seog Lee, Institute of the Mathematical Sciences of the Americas, University of Miami, 1365 Memorial Drive, Ungar 515, Coral Gables, FL 33146, USA}
\email{kyoungseog02@gmail.com}

\author{Han-Bom Moon}
\address{Han-Bom Moon, Department of Mathematics, Fordham University, New York, NY 10023, USA}
\email{hmoon8@fordham.edu}

\maketitle

\begin{abstract}
We prove that the normalized Poincar\'e bundle on the moduli space of stable rank $r$ vector bundles with a fixed determinant on a smooth projective curve $X$ induces a family of nef vector bundles on the moduli space. Two applications follow. We show that when the genus of $X$ is large, the derived category of $X$ is embedded into the derived category of the moduli space for arbitrary rank and coprime degree, which extends the results of Narasimhan, Fonarev-Kuznetsov, and Belmans-Mukhopadhyay. As the second application, we construct a family of ACM bundles on the moduli space. A key ingredient of our proof is the investigation of birational geometry of the moduli spaces of parabolic bundles. 
\end{abstract}

\section{Introduction}\label{sec:introduction}

The study of moduli spaces of vector bundles on curves has a long and beautiful history. The moduli spaces have been central objects in many branches of mathematics, for example, algebraic geometry, differential geometry, mathematical physics, number theory, representation theory, topology, to name a few. This paper discusses the positivity of a restriction of the normalized Poincar\'e bundle on the moduli space and explains two consequences in the study of derived category and arithmetically Cohen-Macaulay (ACM) bundles. 

Let $X$ be a connected smooth projective curve of genus $g \ge 2$. Let $r$, $d$ be two positive integers such that $r \ge 2$ and $(r, d) = 1$. We assume that $0 < d < r$. For a line bundle $L$ of degree $d$ on $X$, let $\rM(r, L)$ be the moduli space of rank $r$, determinant $L$ stable vector bundles. It is a smooth Fano variety of dimension $(r^{2}-1)(g-1)$, and $\Pic(\rM(r, L))$ is generated by an ample divisor $\Theta$. Let $\cE$ be a Poincar\'e bundle over $X \times \rM(r, L)$. For each point $x \in X$, the restriction of $\cE$ to $x \times \rM(r, L) \cong \rM(r, L)$ is denoted by $\cE_{x}$. We assume that $\cE$ is normalized, in the sense that $\det(\cE_{x}) \cong \Theta^{\ell}$ where $\ell$ is the integer such that $0 < \ell < r$ and $\ell d \equiv 1 \; \mbox{mod}\; r$. 

We first show that $\cE_{x}$ and its dual are both nearly positive. 

\begin{theoremletter}\label{thm:nefintro}
Two vector bundles $\cE_{x}$ and $\cE_{x}^{*} \otimes \Theta$ are strictly nef. 
\end{theoremletter}

We observe that this positivity statement has two interesting applications. 

\subsection{Derived category of $\rM(r, L)$}

Our original motivation for this work was to study the bounded derived category $\rD^{b}(\rM(r, L))$ of coherent sheaves on $\rM(r,L)$. The structure of $\rD^{b}(\rM(r, L))$, particularly its semiorthogonal decomposition, has attracted many experts. When $r=2$ and $d=1$, Narasimhan (and independently Belmans-Galkin-Mukhopadhyay) conjectured that $\rD^{b}(\rM(2,L))$ has the following semiorthogonal decomposition, which is known for $g = 2$ case (\cite[Theorem 2.9]{BO95}).

\begin{conjecture}\label{conj:semiorthogonaldecomp}
The derived category of $\rM(2,L)$ has a semiorthogonal decomposition
\[
	\rD^{b}(\rM(2,L)) = \langle \rD^{b}(\mathrm{pt}), \rD^{b}(\mathrm{pt}), \cdots \rD^{b}(X_{k}), \rD^{b}(X_{k}), \cdots, \rD^{b}(X_{g-2}), \rD^{b}(X_{g-2}), \rD^{b}(X_{g-1}) \rangle, 
\]
where $1 \leq k \leq g-2$. Here $X_k$ denotes the $k$-th symmetric product of $X$.
\end{conjecture}

More generally, it has been conjectured that $\rD^{b}(X)$ is embedded into $\rD^{b}(\rM(r, L))$ for every $r \ge 2$. For the normalized Poincar\'e bundle $\cE$ on $X \times \rM(r, L)$, we may consider the Fourier-Mukai transform $\Phi_{\cE} : \rD^{b}(X) \to \rD^{b}(\rM(r, L))$ with the kernel $\cE$. The following conjecture has been a well-known conjecture to experts:

\begin{conjecture}\label{conj:fullyfaithful}
The functor $\Phi_{\cE} : \rD^{b}(X) \to \rD^{b}(\rM(r, L))$ is fully faithful. Therefore, $\rD^{b}(X)$ is embedded into $\rD^{b}(\rM(r, L))$. 
\end{conjecture} 

This is shown for $r = 2$ and $d = 1$ in \cite{Nar17, Nar18, FK18} and for $d = 1$ and $g \ge r + 3$ in \cite{BM19}. This paper proves that the same result holds for any pair $(r, d)$ if they are coprime and if $g$ is larger compared to $r$. 

\begin{theoremletter}\label{thm:mainthm}
Let $r \ge 2$, $d < r$ be two coprime positive integers. If $g \ge r+3$, then $\Phi_{\cE}$ is fully faithful. 
\end{theoremletter}

\begin{remark}\label{rmk:derivedcategory}
\begin{enumerate}
\item Recently, the first author and Narasimhan proved that if $X$ is non-hyperelliptic and $g \ge 16$, then $\rD^{b}(X_{2})$ is embedded into $\rD^{b}(\rM(2, L))$ (\cite{LN21}).
\item For $r = 3$ and $d = 1$, Gomez and the first author provided an explicit conjectural semiorthogonal decomposition (\cite[Conjecture 1.9]{GL20}), and new motivic decompositions of $\rM(r, L)$ for $r \le 3$, compatible with the conjecture, have been found. See \cite{BGM18, Lee18, GL20} and references therein for more details. 
\end{enumerate}
\end{remark}

\subsection{ACM bundles on $\rM(r, L)$}

Our second application is a discovery of a family of ACM bundles on $\rM(r, L)$. For an $n$-dimensional projective variety $V$ with an ample line bundle $A$, a vector bundle $F$ is called \emph{ACM} if $\rH^{i}(V, F \otimes A^{j}) = 0$ for all $j \in \ZZ$ and $0 < i < n$. An ACM bundle $F$ is \emph{Ulrich} if $\rH^{0}(V, F \otimes A^{-1}) = 0$ and $\rH^{0}(V, F) = \mathrm{rank}\; F \cdot \deg V$. ACM bundles naturally appear in matrix factorization (\cite{Eis80}) and correspond to maximal Cohen-Macaulay modules in commutative algebra (\cite{Yos90}). Ulrich bundles enable us to compute their associated Chow forms, and Eisenbud and Schreyer conjectured that every projective variety admits an Ulrich sheaf (\cite{ES03}). However, since the above strong cohomology vanishing is not easy to expect and hard to verify, very few results are known for higher dimensional varieties, even for the existence of ACM bundles (except some trivial examples). 

As a second application of Theorem \ref{thm:nefintro}, we show the following theorem. 

\begin{theoremletter}\label{thm:ACMintro}
Let $r \ge 2$, $0< d < r$ be two coprime positive integers. If $g \ge 3$, $\cE_x$ is an ACM bundle on $\rM(r, L)$ with respect to $\Theta$.
\end{theoremletter}

Indeed, our theorem proves that $\cE_{x}$ is ACM with respect to every ample line bundle on $\rM(r, L)$ -- see Definition \ref{def:ACM} and Remark \ref{rmk:veryample}.

Our proof does not cover $g = 2$ case and it seems that it requires a new technique. We expect that $\cE_{x}$ is still ACM in this case. It will be an interesting problem to verify it. Note that $\cE_{x}$ is an ACM bundle when $g=r=2$ (\cite{CKL19, FK18}). 

\subsection{Sketch of proof}

The key ingredient of the proof is a study of birational geometry of the moduli space of parabolic bundles. The moduli space $\rM(r, L, \bm, \ba)$ (see Section \ref{sec:parabolic} for the definition and notation) of parabolic bundles depends on the choice of stability condition $\ba$ and the analysis of wall-crossing has been studied well (Section \ref{sec:wallcrossing}). Moreover, the birational geometry of $\rM(r, L, \bm, \ba)$ is governed by the wall-crossing: Every rational contraction that appears in Mori's program can be described in terms of wall-crossings or their degenerations -- the forgetful map (Example \ref{ex:forgetfulmap}) and generalized Hecke correspondences (Remark \ref{rmk:Hecke}). Consult \cite{MY20, MY21} to see a more general framework. 

\subsubsection{The positivity of $\cE_{x}$ and $\cE_{x}^{*}\otimes \Theta$}

The first observation is that $\PP(\cE_{x})$ is isomorphic to the moduli space $\rM(r, L, r-1, \epsilon)$ with sufficiently small parabolic weight $\epsilon > 0$. From the wall-crossing diagram, we can obtain two extremal rays of the nef cone of $\PP(\cE_{x})$. Theorem \ref{thm:nefintro} follows from the intersection computation on $\PP(\cE_{x})$ in Section \ref{sec:nef}. 

\subsubsection{Embedding of the derived category}

Bondal-Orlov criterion (\cite[Theorem 1]{BO95}, Theorem \ref{thm:vanishing}) deduces Theorem \ref{thm:mainthm} to a problem checking the vanishing of cohomologies of the form $\rH^{i}(\rM(r, L), \cE_{x} \otimes \cE_{y}^{*})$. We relate this vanishing with vanishing of a certain line bundle on $\PP(\cE_{x}) \times_{\rM(r, L)}\PP(\cE_{y}^{*})$, which is identified with $\rM(r, L, (r-1, 1), (\epsilon, \epsilon))$. In Section \ref{sec:derivedcategory}, we study the birational geometry of the latter space. We compute the effective cone and show that it is of Fano type. Then the vanishing follows from Kawamata-Viehweg vanishing and Le Potier vanishing theorems, and cohomology extensions. 

\subsubsection{ACM bundles}

Since the ACM condition is a cohomology vanishing condition, we may apply the same technique. We replace the vanishing of the cohomology of $\cE_{x} \otimes \Theta^{j}$ on $\rM(r, L)$ with that of line bundles on $\PP(\cE_{x}) \cong \rM(r, L, r-1, \epsilon)$. Then the above-mentioned technique can be applied in this setup, and we verity the vanishing in Section \ref{sec:ACM}.

\subsection*{Conventions}
\begin{itemize}
\item We work over $\CC$. 
\item In this paper, $X$ denotes a smooth connected projective curve of genus $g \ge 2$. 
\item The coarse moduli space of rank $r$, determinant $L$ semistable vector bundles on $X$ is denoted by $\rM(r, L)$. The degree of $L$ is denoted by $d$. Unless stated explicitly, we assume that $(r, d) = 1$, so $\rM(r, L)$ is a smooth projective variety. 
\item Let $\Theta$ be the ample generator on $\Pic(\rM(r, L))$. 
\item Let $\cE$ be the normalized Poincar\'e bundle on $X \times \rM(r, L)$ such that for each $x \in X$, its restriction $\cE_{x}$ to $x \times \rM(r, L) \cong \rM(r, L)$ has the determinant $\Theta^{\ell}$. Here $\ell$ is a unique integer satisfying $\ell d \equiv 1 \;\mathrm{mod}\; r$ and $0 < \ell < r$. 
\item For a vector space $W$, $\PP(W)$ is the projective space of one-dimensional quotients. 
\item For a variety $V$, the bounded derived category of coherent sheaves on $V$ is denoted by $\rD^b(V)$. 
\item Every algebraic stack is defined over the fppf topology.
\end{itemize}

\subsection*{Acknowledgements}

The authors thank M. S. Narasimhan for drawing their attention to this problem, sharing his idea, and providing valuable suggestions about this and related projects. Especially, the first author would like to express his deepest gratitude to him for invaluable teachings and warm encouragements for many years. Part of this work was done when the first author was visiting Indian Institute of Science (Bangalore) where he enjoyed wonderful working conditions. He thanks Gadadhar Misra for kind hospitality during his stay IISc. He also thanks Ludmil Katzarkov and Simons Foundation for partially supporting this work via Simons Investigator Award-HMS.


\section{Parabolic bundles and their moduli spaces}\label{sec:parabolic}

In this section, we introduce the notion of parabolic vector bundles and their moduli space. This paper only considers the parabolic structure consisting of at most one flag for each parabolic point. Fix a smooth connected projective curve $X$ and a finite ordered set $\bp := (p_{1}, p_{2}, \cdots, p_{k})$ of distinct closed points of $X$. 

\begin{definition}\label{def:parabolicbundle}
A \emph{parabolic bundle} over a pointed curve $(X, \bp)$ of rank $r$ is a collection of data $(E, V_{\bullet})$ where
\begin{enumerate}
\item $E$ is a rank $r$ vector bundle over $X$;
\item $V_{\bullet} = (V_{1}, V_{2}, \cdots, V_{k})$ where $V_{i}$ is a subspace of $E|_{p_{i}}$. The dimension of $V_{i}$ is called the \emph{multiplicity} of $V_{i}$ and denoted by $m_{i}$. 
\end{enumerate}
The sequence $\bm = (m_{1}, m_{2}, \cdots, m_{k})$ is called the \emph{multiplicity} of the parabolic bundle $(E, V_{\bullet})$. 
\end{definition}

\begin{definition}\label{def:modulistackofparabolicbundle}
Let $\cM_{(X, \bp)}(r, d, \bm)$ (resp. $\cM_{(X, \bp)}(r, L, \bm)$) be the moduli stack of parabolic bundles $(E, V_{\bullet})$ over $(X, \bp)$ of rank $r$, degree $d$ (resp. determinant $L$), and multiplicity $\bm$. If there is no confusion, we use $\cM(r, d, \bm)$ (resp. $\cM(r, L, \bm)$). 
\end{definition}

It is straightforward to see that $\cM(r, L, \bm)$ is a $\times \Gr(m_{p_{i}}, r)$-bundle over $\cM(r, L)$, the stack of all vector bundles of rank $r$ and determinant $L$. In particular, this Artin stack is highly non-separated. To obtain a projective coarse moduli space that enables us to do projective birational geometry, we need to introduce a stability condition. 

For a parabolic bundle $(E, V_{\bullet})$, a \emph{parabolic subbundle} $(F, W_{\bullet})$ is a pair such that $F \subset E$ is a subbundle and $W_{i} = F|_{i}\cap V_{i}$. A \emph{parabolic quotient bundle} is defined as a parabolic bundle $(E/F, Y_{\bullet})$ such that $Y_{i} = \im (V_{i} \to E/F|_{i})$. 

A \emph{parabolic weight} $\ba = (a_{1}, a_{2}, \cdots, a_{k})$ is a sequence of rational numbers such that $0 < a_{i} < 1$. Intuitively, we may regard $\ba$ as extra weight for the parabolic flags. For a parabolic bundle $(E, V_{\bullet})$, its \emph{parabolic degree} is $\pdeg (E, V_{\bullet}) := \deg E + \sum_{1\le i \le k}m_{i}a_{i}$. The same parabolic weight can induce the parabolic degree for parabolic subbundles and parabolic quotient bunddles of $(E, V_{\bullet})$. The \emph{parabolic slope} is $\mu(E, V_{\bullet}) := \pdeg (E, V_{\bullet})/\rk E$. 

\begin{definition}\label{def:stability}
Fix a parabolic weight $\ba$. A parabolic bundle $(E, V_{\bullet})$ is \emph{$\ba$-(semi)stable} if for every parabolic subbundle $(F, W_{\bullet})$, $\mu(F, W_{\bullet}) \;(\le ) < \mu(E, V_{\bullet})$. A parabolic weight $\ba$ is \emph{general} if the $\ba$-semistability coincides with the $\ba$-stability. 
\end{definition}

\begin{definition}\label{def:modulispaceparabolicbundle}
Let $(X, \bp)$ be a $k$-pointed curve of genus $g \ge 2$. Let $\cM(r, d, \bm, \ba)$ (resp. $\cM(r, L, \bm, \ba)$) be the moduli stack of rank $r$, degree $d$ (resp. determinant $L$), $\ba$-semistable parabolic bundles over $(X, \bp)$. Let $\rM(r, d, \bm, \ba)$ (resp. $\rM(r, L, \bm, \ba)$) be its good moduli space, which is a normal projective variety of dimension $r^{2}(g-1)+1 + \sum m_{i}(r-m_{i})$ (resp. $(r^{2}-1)(g-1) + \sum m_{i}(r-m_{i})$). When $\ba$ is general, both $\rM(r, d, \bm, \ba)$ and $\rM(r, L, \bm, \ba)$ are nonsingular.
\end{definition}

\begin{remark}
When $g \le 1$, the moduli space behaves differently. For instance, if $g = 0$, depending on $\ba$, $\cM(r, L, \bm, \ba)$ may be empty. Consult \cite{MY21}. 
\end{remark}

\begin{example}\label{ex:smallweight}
The inequality $\mu(F, W_{\bullet}) \le \mu(E, V_{\bullet})$ defining the semistability can be understood as a perturbation of the inequality $\mu(F) \le \mu(E)$ for the semistability of the underlying bundle. If $(r, d = \det L) = 1$ and each coefficient of $\ba$ is sufficiently small and general, then the parabolic weight does not affect on the stability. Therefore, a parabolic bundle $(E, V_{\bullet})$ is $\ba$-stable if and only if the underlying bundle $E$ is stable. Thus, there is a forgetful morphism $\cM(r, L, \bm, \ba) \to \cM(r, L)$ and that between coarse moduli spaces 
\[
	\pi : \rM(r, L, \bm, \ba) \to \rM(r, L)
\]
and $\pi$ is a $\times \Gr(m_{i}, r)$-fibration. Indeed, for a fixed Poincar\'e bundle $\cE$ over $X \times\rM(r, L)$, 
\[
	\rM(r, L, \bm, \ba) \cong \times_{\rM(r, L)}\Gr(m_{i}, \cE_{p_{i}}).
\]
\end{example}

\begin{example}\label{ex:forgetfulmap}
More generally, if $\ba = (a_{i})$ is general and one $a_{i}$ is sufficiently small, then forgetting one flag does not affect on the stability calculation. Thus, there is a forgetful morphism 
\[
	\pi : \rM_{(X, \bp)}(r, L, \bm, \ba) \to \rM_{(X, \bp')}(r, L, \bm', \ba')
\]
where $\bp' = \bp \setminus \{p_{i}\}$, $\bm' = \bm \setminus \{m_{i}\}$, and $\ba' = \ba \setminus \{a_{i}\}$. This is a $\Gr(m_{i}, r)$-fibration. 
\end{example}

\begin{example}\label{ex:topandzerodimflags}
Fix a $k$-pointed curve $(X, \bp)$. Let $\bp' := \bp \setminus \{p_{k}\}$. Let $\bm' = (m_{i})_{1 \le i \le k-1}$ and $\ba' = (a_{i})_{1 \le i \le k-1}$. Suppose that $m_{k} = 0$ or $r$. Then 
\[
	\rM_{(X, \bp)}(r, L, \bm, \ba) \cong \rM_{(X, \bp')}(r, L, \bm', \ba'). 
\]
\end{example}

When one of the parabolic weights is sufficiently close to one, there is another contraction morphism. 

\begin{proposition}\label{prop:generalizedHeckemodification}
We use the notation in Example \ref{ex:topandzerodimflags}. For a general parabolic weight $\ba = (a_{i})$, assume that $a_{k}$ is sufficiently close to one. Then there exists a functorial morphism 
\[
	\pi_{1} : \rM_{(X, \bp)}(r, L, \bm, \ba) \to \rM_{(X, \bp')}(r, L(-(r-m_{k})p_{k}), \bm', \ba').
\]
\end{proposition}

\begin{proof}
It is sufficient to construct a morphism 
\[
	\cM_{(X, \bp)}(r, L, \bm, \ba) \to \cM_{(X, \bp')}(r, L(-(r-m_{k})p_{k}), \bm', \ba')
\]
between algebraic stacks. 

Let $\widetilde{\bm} = (\widetilde{m}_{i})$ be a multiplicity such that $\widetilde{m}_{i} = m_{i}$ for $1 \le i \le k-1$ and $\widetilde{m}_{k} = r$. By Example \ref{ex:topandzerodimflags}, there is a functorial isomorphism $\cM_{(X, \bp)}(r, L(-(r-m_{k})p_{k}), \widetilde{\bm}, \ba) \cong \cM_{(X, \bp')}(r, L(-(r-m_{k})p_{k}), \bm', \ba')$. Thus it is sufficient to show that there is a morphism 
\[
	\cM(r, L, \bm, \ba) \to \cM(r, L(-(r-m_{k})p_{k}), \widetilde{\bm}, \ba).
\]

For a stable bundle $(E, V_{\bullet}) \in \cM(r, L, \bm, \ba)$, let $E'$ be the kernel of the restriction map $E \to E|_{p_{k}} \to E|_{p_{k}}/V_{p_{k}}$. Then for each $i \ne k$, $E'|_{p_{i}}$ can be identified with $E|_{p_{i}}$. Set $V_{i}' = V_{i}$ under this identification. On the other hand, the restriction $f : E'|_{p_{k}} \to E|_{p_{k}}$ is a linear map with image $V_{k}$. We set $V_{k}' := f^{-1}(V_{k}) = E'|_{p_{k}}$. Then we obtain a parabolic bundle $(E', V_{\bullet}') \in \cM(r, L(-(r-m_{k})p_{k}), \widetilde{\bm}, \ba)$. Thus, we have a morphism 
\begin{equation}\label{eqn:generalizedHecke}
\begin{split}
\cM(r, L, \bm, \ba) &\to \cM(r,L(-(r-m_{k})p_{k}), \widetilde{\bm})\\
(E, V_{\bullet}) & \mapsto (E', V_{\bullet}').
\end{split}
\end{equation}

We claim that $(E', V_{\bullet}')$ is $\ba$-semistable. Then the morphism in Equation \eqref{eqn:generalizedHecke} factors through $\cM(r,L(-(r-m_{k})p_{k}), \widetilde{\bm}, \ba)$. 

Suppose not. Then there is a parabolic subbundle $(F', W_{\bullet}')$ of $(E', V_{\bullet}')$ such that $\mu(F', W_{\bullet}') > \mu(E', V_{\bullet}')$. Let $\rk F' = s$, $\deg F' = e$, and $n_{i} = \dim W_{i}'$. Note that $n_{k} = s$. 

Set $d = \det L$. Then
\begin{equation}\label{eqn:slopecomparison}
\begin{split}
	\mu(E, V_{\bullet}) - \mu(E', V_{\bullet}') &= \frac{d + \sum m_{i}a_{i}}{r} - \frac{d - (r-m_{k}) + \sum_{i \ne k}m_{i}a_{i} + ra_{k}}{r}\\ 
	&= \frac{(r-m_{k})(1-a_{k})}{r}.
\end{split}
\end{equation}

In general, $F'$ is not a subbundle of $E$. But there is a subbundle $F$ of $E$ such that $F/F'$ is a sheaf supported on $p_{k}$ and $\dim (F/F')|_{p_{k}} = s - c$, where $c := \dim F|_{p_{k}} \cap V_{p_{k}}$. For the induced parabolic subbundle $(F, W_{\bullet})$ of $(E, V_{\bullet})$, 
\begin{equation}\label{eqn:slopecomparisonsubbundle}
\begin{split}
	\mu(F, W_{\bullet}) - \mu(F', W_{\bullet}') &= \frac{e + (s-c) + \sum_{i\ne k}a_{i}n_{i} + a_{k}c}{s} - \frac{e + \sum_{i \ne k}a_{i}n_{i} + a_{k}s}{s}\\
	&= \frac{(s-c)(1-a_{k})}{s}.
\end{split}
\end{equation}

By combining \eqref{eqn:slopecomparison} and \eqref{eqn:slopecomparisonsubbundle}, we have 
\[
	\mu(E, V_{\bullet}) - \mu(F, W_{\bullet}) = \mu(E', V_{\bullet}') - \mu(F', W_{\bullet}') + (1-a_{k})\left(\frac{r-m_{k}}{r} - \frac{s - c}{s}\right).
\]
Note that $\mu(E', V_{\bullet}') - \mu(F', W_{\bullet}')$ is independent from $a_{k}$, as the coefficient of $a_{k}$ in each term is one. Thus, if $a_{k}$ is sufficiently close to one, then the last term is negligible. By the assumption, $\mu(E', V_{\bullet}') - \mu(F', W_{\bullet}') < 0$ and hence the left hand side is also negative. It violates the stability of $(E, V_{\bullet})$ and obtain a contradiction. 
\end{proof}

\begin{remark}\label{rmk:Hecke}
The morphism in Proposition \ref{prop:generalizedHeckemodification} can be understood as a generalized Hecke correspondence. When $d = k = 1$ and $m = r-1$, up to taking a dual bundle, we obtain the classical Hecke correspondence in the sense of \cite[Section 4]{NR75}. A difference in $d > 1$ case is that $\rM(r, L, m, a)$ does not admit morphisms to both $\rM(r, L)$ and $\rM(r, L(-x))$, so we need a birational modification on $\rM(r, L, m, a)$. It can be explained in terms of parabolic wall-crossing, which will be explained in Section \ref{sec:wallcrossing} below. 
\end{remark}

\section{Wall-crossing}\label{sec:wallcrossing}

In this section, we review how the moduli space $\bM(r, L, \bm, \ba)$ changes as $\ba$ varies. 

\subsection{General theory}

Let $k$ be the number of parabolic points. Recall that a parabolic weight is, under our restrictive setting that each parabolic point has a single parabolic flag, a length $k$ sequence of rational number $\ba = (a_{i})$ with $0 < a_{i} < 1$. The closure of the set of parabolic weights is $[0, 1]^{k} \subset \RR^{k}$. 

There is a wall-chamber decomposition of $[0, 1]^{k}$. A parabolic bundle $(E, V_{\bullet}) \in \bM(r, L, \bm, \ba)$ is strictly semi-stable if and only if there is a maximal destabilizing subbundle $(F, W_{\bullet})$ such that $\mu(F, W_{\bullet}) = \mu(E, V_{\bullet})$. More explicitly, this is true only if
\begin{equation}\label{eqn:wallequation}
	\frac{e+\sum n_{i}a_{i}}{s} = \frac{d+\sum m_{i}a_{i}}{r}
\end{equation}
for some $0 < s < r$, $e \in \ZZ$, and $\bn = (n_{i})$. Here $s$ is the rank, $d$ is the degree, and $\bn$ is the multiplicity of $(F, W_{\bullet})$. So we require that $n_{i} \le \mathrm{min}\;\{s, m_{i}\}$. Let $\Delta(s, e, \bn)$ be the set of weights that satisfy \eqref{eqn:wallequation}. Note that this is an intersection of a hyperplane and $[0, 1]^{k}$. We call $\Delta(s, e, \bn)$ a \emph{wall} if it is nonempty. We also obtain 
\begin{equation}\label{eqn:wallduality}
	\Delta(s, e, \bn) = \Delta(r-s, d-e, \bm - \bn).
\end{equation}
Note that $\Delta(s, e, \bn) = \Delta(ks, ke, k\bn)$ if $ks < r$ for some $k > 1$. We call such a wall a \emph{multiple wall}, and otherwise, it is a \emph{simple wall}. A wall $\Delta(s, e, \bn)$ is simple if and only if $\{s, e, n_{i}\}$ are coprime and $\{r-s, d-e, m_{i} - n_{i}\}$ are coprime. 

The stability changes only if a parabolic weight $\ba$ lies on one of the walls. So for each open chamber $C \subset [0, 1]^{k}$, for any $\ba, \ba' \in C$, $\rM(r, L, \bm, \ba) \cong \rM(r, L, \bm, \ba')$. The stability coincides with the semistability if $\ba \in (0, 1)^{k} \setminus \bigcup \Delta(s, e, \bn)$. 

Let 
\[
	\xymatrix{\bM(r, L, \bm, \ba^{-}) \ar@{<-->}[rr] \ar[rd]_{\pi_{-}}&& \bM(r, L, \bm, \ba^{+}) \ar[ld]^{\pi_{+}}\\
	&\bM(r, L, \bm, \ba)}
\]
be a wall-crossing. Suppose that $\ba$ is a general point of $\Delta(s, e, \bn)$, and $\ba^{-}$ and $\ba^{+}$ are two very close weights on the opposite chambers. The contraction maps $\pi_{\pm}$ are birational surjections. Let $Y^{\pm}$ be the exceptional locus on $\bM(r, L, \bm, \ba^{\pm})$ and let $Y := \pi_{\pm}(Y^{\pm})$. The subvarieties $Y^{\pm}$ are called the \emph{wall-crossing centers}. For $(E, V_{\bullet}) \in Y^{+}$, there is a unique maximal destabilizing subbundle $(E^{-}, V_{\bullet}^{-}) \in \rM(s, e, \bn, \ba)$, which fits into an exact sequence 
\[
	0 \to (E^{-}, V_{\bullet}^{-}) \to (E, V_{\bullet}) \to (E^{+}, V_{\bullet}^{+}) \to 0
\]
of parabolic bundles. The map $\pi_{-}$ is restricted to the map $Y^{-} \to Y$, which sends $(E, V_{\bullet})$ to $((E^{-}, V_{\bullet}^{-}), (E^{+}, V_{\bullet}^{+}))$. Conversely, if $ x := ((E^{-}, V_{\bullet}^{-}), (E^{+}, V_{\bullet}^{+}))$ is a general point in $Y$ so that both $(E^{-}, V_{\bullet}^{-})$ and $(E^{+}, V_{\bullet}^{+})$ are stable, then the fiber $\pi_{-}^{-1}(x)$ is a projective space $\PP \Ext^{1}((E^{+}, V_{\bullet}^{+}), (E^{-}, V_{\bullet}^{-}))$ (\cite[Lemma 1.4]{Yok95}). If $(E, V_{\bullet})$ is in a unique irreducible component of $Y^{-}$, the image of the component is isomorphic to $\rM(s, e, \bn, \ba) \times_{\mathrm{Pic}(X)}\rM(r-s, d-e, \bm - \bn, \ba)$. 

For our purpose, we need a lower bound of the codimension of $Y^{\pm}$. Observe that the parabolic bundles in $Y^{-}$ are stable with respect to $\ba^{-}$, but unstable with respect to $\ba^{+}$. Thus, $Y^{-}$ parametrizes unstable parabolic bundles with respect to some weight. The codimension of an unstable locus is estimated in \cite{Sun00}. For an outline of the proof, see also \cite[Section 3.2]{MY20}.

\begin{theorem}[\protect{\cite[Proposition 5.1]{Sun00}}]\label{thm:codimunstable}
In $\cM(r, L, \bm)$, the codimension of the unstable locus is at least $(r-1)(g-1) + 1$. 
\end{theorem}

\begin{corollary}\label{cor:codimcenter}
The codimension of the wall-crossing center is at least $(r-1)(g-1) + 1$. In particular, if $g \ge 2$, every wall-crossing is a flip. 
\end{corollary}

We say a wall-crossing is \emph{simple} if:
\begin{enumerate}
\item The wall $\Delta(s, e, \bn)$ is a simple wall and;
\item $\ba \in \Delta(s, e, \bn)$ is on a unique wall. 
\end{enumerate}
A simple wall-crossing has an explicit description. The wall-crossing centers $Y^{\pm}$ are irreducible and their image $Y \cong \rM(s, e, \bn, \ba) \times_{\mathrm{Pic}(X)}\rM(r-s, d-e, \bm - \bn, \ba)$ is a smooth variety. Let $(\cE^{-}, \cV_{\bullet}^{-})$ (resp. $(\cE^{+}, \cV_{\bullet}^{+})$) be the Poincar\'e family over $\rM(s, e, \bn, \ba)$ (resp. $\rM(r-s, d-e, \bm- \bn, \ba)$). The standard GIT construction and the descent method imply the existence of Poincar\'e bundle (\cite[Chapter 5]{New78}, \cite[Section 4.6]{HL10}). Then $Y^{-} \cong \PP R^{1}\pi_{- *}\CParHom((\cE^{+}, \cV_{\bullet}^{+}), (\cE^{-}, \cV_{\bullet}^{-}))$ and $Y^{+} \cong \PP R^{1}\pi_{+ *}\CParHom((\cE^{-}, \cV_{\bullet}^{-}), (\cE^{+}, \cV_{\bullet}^{+}))$ (\cite[Section 1]{Yok95}). In particular, they are projective bundles over $Y$. Finally, it is well-known that the blow-up of $\bM(r, L, \bm, \ba^{-})$ along $Y^{-}$ is isomorphic to the blow-up of $\bM(r, L, \bm, \ba^{+})$ along $Y^{+}$. 

\subsection{Main example}

From now on, we focus on one particular case where $k = 2$ and $\bm = (r-1, 1)$, which is our primary interest in this paper. We set $\bp = (x, y)$ and use an appropriate modification of the notation such as $\bm = (m_{x}, m_{y})$ and $\ba = (a_{x}, a_{y})$. 

Let $\Delta(s, e, \bn)$ be a wall on $[0,1]^{2}$ and let $\ba = (a_{x}, a_{y})$ be a general point on it. Let $(E, V_{\bullet}) \in Y \subset \rM(r, L, \bm, \ba)$ be a general polystable parabolic bundle on the wall-crossing center. Then $(E, V_{\bullet}) \cong (F_{1}, W_{1 \bullet}) \oplus (F_{2}, W_{2\bullet})$ and $\mu(E, V_{\bullet}) =\mu(F_{1}, W_{1\bullet}) = \mu(F_{2}, W_{2\bullet})$.

There are two possibilities. First of all, it is possible that one of $F_{i}$'s (say $F_{1}$) has the largest possible intersection with the flags of $E$. That means, $\dim F_{1}|_{x} \cap V_{x} = \dim F_{1}|_{x} = s$ and $\dim F_{1}|_{y} \cap V_{y} = \dim V_{y} = 1$. Then we have an equality 
\[
	\frac{e+sa_{x} + a_{y}}{s} = \frac{d + (r-1)a_{x} + a_{y}}{r},
\]
or equivalently, $sa_{x} + (r-s)a_{y} = sd - re$. The slope of the line on the $(a_{x}, a_{y})$-plane is negative, so we will call the wall a \emph{negative wall}. To intersect with the interior of $[0, 1]^{2}$, it is necessary that $0 < sd - re < r$. Since these walls are $\Delta(s, e, (s, 1)) = \Delta(r-s, d-e, (r-s-1, 0))$, they are simple walls. 

The second case is that $F_{1}$ has the maximum intersection with the flag on $x$, but does not intersect on $y$. In other words, $\dim F_{1}|_{x} \cap V_{x} = \dim F_{1}|_{x}$ and $\dim F_{1}|_{y} \cap V_{y} = 0$. Then we have
\[
	\frac{e+sa_{x}}{s} = \frac{d+(r-1)a_{x} + a_{y}}{r}, 
\]
so $sa_{x} - sa_{y} = sd - re$. Thus, the slope of the wall $\Delta(s, e, (s, 0))$ is one and we call it a \emph{positive wall}. The nonempty intersection with $(0, 1)^{2}$ is equivalent to $-s < sd - re < s$. Since $(r, d) = 1$, $sd - re \ne 0$ and there is no wall passing through the origin. 

See Figure \ref{fig:wallchamber} for an example of the wall-chamber decomposition.

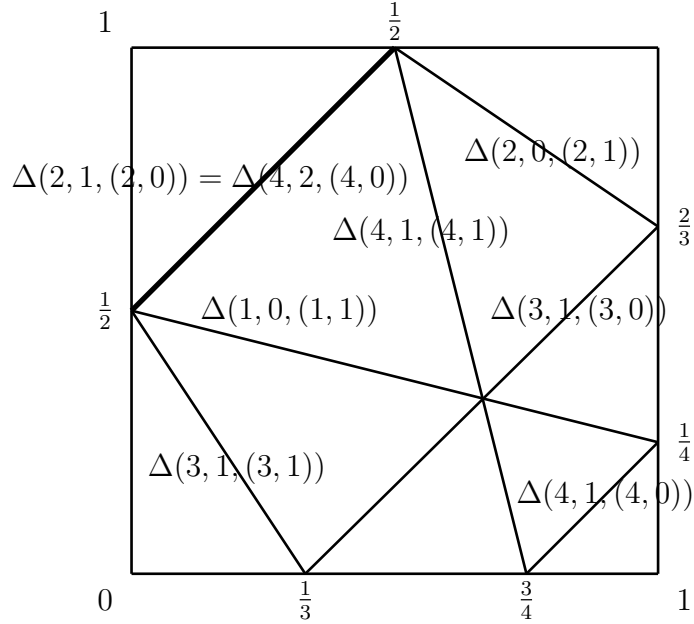
\begin{figure}[!ht]
\begin{tikzpicture}[scale=7]
	\draw[line width = 1pt] (0, 0) -- (1, 0);
	\draw[line width = 1pt] (1, 0) -- (1, 1);
	\draw[line width = 1pt] (0, 0) -- (0, 1);
	\draw[line width = 1pt] (0, 1) -- (1, 1);
	\draw[line width = 1pt] (0.33, 0) -- (0, 0.5);
	\draw[line width = 1pt] (1, 0.25) -- (0, 0.5);
	\draw[line width = 1pt] (0.75, 0) -- (0.5, 1);
	\draw[line width = 1pt] (1, 0.66) -- (0.5, 1);
	\draw[line width = 1pt] (0.75, 0) -- (1, 0.25);
	\draw[line width = 2pt] (0, 0.5) -- (0.5, 1);
	\draw[line width = 1pt] (0.33, 0) -- (1, 0.66);

	\node at (-0.05, -0.05) {$0$};
	\node at (1.05, -0.05) {$1$};
	\node at (-0.05, 1.05) {$1$};
	\node at (-0.05, 0.5) {$\frac{1}{2}$};
	\node at (0.33, -0.05) {$\frac{1}{3}$};
	\node at (0.75, -0.05) {$\frac{3}{4}$};
	\node at (1.05, 0.25) {$\frac{1}{4}$};
	\node at (1.05, 0.66) {$\frac{2}{3}$};
	\node at (0.5, 1.05) {$\frac{1}{2}$};
	\node at (0.2, 0.2) {$\Delta(3, 1, (3, 1))$};
	\node at (0.3, 0.5) {$\Delta(1, 0, (1, 1))$};
	\node at (0.55, 0.65) {$\Delta(4, 1, (4, 1))$};
	\node at (0.8, 0.8) {$\Delta(2, 0, (2, 1))$};
	\node at (0.15, 0.75) {$\Delta(2, 1, (2, 0)) = \Delta(4, 2, (4, 0))$};
	\node at (0.85, 0.5) {$\Delta(3, 1, (3, 0))$};
	\node at (0.9, 0.15) {$\Delta(4, 1, (4, 0))$};
\end{tikzpicture}
\caption{The wall-chamber decomposition for $r = 5$ and $d = 2$}\label{fig:wallchamber}
\end{figure}

\subsection{Mori's program}

The wall-crossing picture can be incorporated with projective birational geometry of $\rM(r, L, \bm, \ba)$ in the nicest way. Let $\ba$ be a general parabolic weight. Then every rational contraction of $\rM(r, L, \bm, \ba)$ can be obtained in terms of wall-crossings, forgetful maps, and generalized Hecke correspondences. 

Recall that $\mathrm{Pic}^{G}(V)$ is the group of linearized line bundles on $V$. Let $\rN^{1, G}(V)_{\RR}$ be the numerical classes of linearized $\RR$-line bundles. 

\begin{lemma}\label{lem:GIT}
Let $G$ be a reductive group. Let $V$ be a normal $\QQ$-factorial projective variety equipped with a linearized $G$-action. Let $L \in \mathrm{Pic}^{G}(V)$ be a linearized ample divisor such that $V^{ss}(L) = V^{s}(L) \ne \emptyset$. Then there is a surjective linear map $\rN^{1, G}(V)_{\RR} \to \rN^{1}(V\git_{L}G)_{\RR}$. 
\end{lemma}

\begin{proof}
Let $E \in \rN^{1,G}(V)_{\QQ}$. Then $E$ is represented by a linearized $\QQ$-line bundle $E$. By taking some power, we may assume that $E$ is a genuine line bundle. The coincidence of the stability and the semistability implies that for each point $x \in V^{ss}(L)$, the stabilizer is a finite group. Thus, if we take some power again, we may assume that the stabilizer acts on each fiber of $E$ trivially. Now by Kempf's descent lemma (\cite[Theorem 2.3]{DN89}), $E$ descends to a line bundle $E\git_{L}G$ over $V\git_{L}G$. This map can be linearly extended and completed, so we have a desired linear map $\rN^{1, G}(V)_{\RR} \to \rN^{1}(V \git_{L}G)_{\RR}$. It is surjective, since for any line bundle $F$ on $V\git_{L}G$, its pull-back $\widetilde{F}$ on $V^{ss}(L)$ is a line bundle. By the $\QQ$-factoriality of $V$, after taking some power, it can be extended to a line $\widetilde{F}$ bundle on $V$ (not necessarily uniquely, depending of the codimension of $V \setminus V^{ss}(L)$). Since $V$ is normal, some power of $\widetilde{F}$ admits a linearization (\cite[Corollary 1.6]{MFK94}). So we obtain an element in $\rN^{1, G}(V)_{\RR}$. 
\end{proof}

\begin{proposition}\label{prop:weightanddivisor}
Let $[0, 1]^{k}$ be the closure of the set of parabolic weights. Let $\ba \in (0, 1)^{k}$ be a general parabolic weight. Then there is a linear isomorphism between a cone over $[0, 1]^{k}$ and the effective cone $\Eff(\rM(r, L, \bm, \ba))$ of divisors. 
\end{proposition}

\begin{proof}
By the standard construction of the moduli space of parabolic bundles as an $\SL_{N}$ GIT quotient (\cite[Section 4]{MS80}), all of them can be constructed as a GIT quotient of the same smooth variety $Z$ with various linearizations, and the parabolic weights depend linearly on the choice of linearization. In particular, there is a linear embedding $(0, 1)^{k} \to \rN^{1, \SL_{N}}(Z)_{\RR}$. Lemma \ref{lem:GIT} induces a linear embedding $(0, 1)^{k} \to \rN^{1}(Z\git_{L}\SL_{N})_{\RR} = \rN^{1}(\rM(r, L, \bm, \ba))_{\RR}$. Since $Z$ is normal, the forgetful map $\rN^{1, \SL_{N}}(Z)_{\RR} \to \rN^{1}(Z)_{\RR}$ is surjective (\cite[Corollary 1.6]{MFK94}). Since the character group of $\SL_{N}$ is trivial, it is injective. So $\rN^{1, \SL_{N}}(Z)_{\RR} = \rN^{1}(Z)_{\RR}$. The map $\rN^{1}(Z)_{\RR} \to \rN^{1}(Z\git_{L}\SL_{N})_{\RR}$ is bijective because the unstable locus has codimension $\ge 2$ (Theorem \ref{thm:codimunstable}). Therefore the map $(0, 1)^{k} \to \rN^{1}(\rM(r, L, \bm, \ba))_{\RR}$ is also a linear embedding. Thus, there is an embedding of the cone over $(0, 1)^{k}$ to $\Eff(\rM(r, L, \bm, \ba))$ which maps $\ba'$ to its associated line bundle $L_{\ba'}$. 

Now we show that the cone over the closure $[0, 1]^{k}$ of $(0, 1)^{k}$ can be identified with $\Eff(\rM(r, L, \bm, \ba))$. Recall that for any effective divisor $D$ (or equivalently, a line bundle $\cO(D)$) of a normal $\QQ$-factorial projective variety $V$, we may associate a rational contraction $V \dashrightarrow V(D)$ where 
\[
	V(D) := \proj \bigoplus_{m \ge 0}\rH^{0}(V, \cO(mD)).
\]
Conversely, any rational contraction of $V$ can be obtained in this way. If $D \in \mathrm{int}\;\Eff(V)$, then $V \dashrightarrow V(D)$ is a birational map and if $D \in \partial \Eff(V)$, $V \dashrightarrow V(D)$ is a contraction with positive dimensional general fibers. 

Note that on the boundary $\partial [0, 1]^{k}$, one of the coordinates must be either zero or one. In the first case, we can obtain a rational contraction $\rM(r, L, \bm, \ba) \to \rM(r, L, \bm', \ba')$ in Example \ref{ex:forgetfulmap}. In the latter case, we have a generalized Hecke modification in Proposition \ref{prop:generalizedHeckemodification}. All of them are contractions with positive dimensional fibers, so they must be associated with divisors on the boundary of the effective cone. Since the effective cone is convex, it is sufficient to obtain the result. 
\end{proof}

\section{Nef vector bundles}\label{sec:nef}

Let $\cE$ be the normalized Poincar\'e bundle over $X \times \rM(r, L)$. Recall that for any $x \in X$, $\cE_{x}$ is the vector bundle on $\rM(r, L)$ obtained by restricting $\cE$ on $x \times \rM(r, L)$. In this section, we prove the nefness of $\cE_{x}$. 

\begin{theorem}\label{thm:nef}
The restricted Poincar\'e bundle $\cE_{x}$ is a strictly nef vector bundle. 
\end{theorem}

\begin{remark}\label{rmk:nef}
The case of $d = 1$ of Theorem \ref{thm:nef} is shown in \cite[Proposition 3.3]{Nar17} and \cite[Lemma 13]{BM19}. So we assume that $d > 1$. Consult Remark \ref{rmk:d=1} to check the difference for $d = 1$ case. 
\end{remark}

We obtain another strictly nef bundle immediately. This proves Theorem \ref{thm:nefintro}.

\begin{corollary}\label{cor:nef}
The vector bundle $\cE_{x}^{*}\otimes \Theta$ is strictly nef. 
\end{corollary}

\begin{proof}
Fix a line bundle $A$ of degree $1$ on $X.$ Consider the vector bundle $\cE^* \otimes p^* A \otimes q^* \Theta$ on $X \times \rM(r,L)$, where $p : X \times \rM(r, L) \to X$ and $q : X \times \rM(r, L) \to \rM(r, L)$ are two projections.  From the isomorphism $\rM(r,L) \cong \rM(r,L^{*}) \cong \rM(r, A^{r} \otimes L^{*}),$ we see that $\cE^* \otimes p^* A\otimes q^* \Theta$ is the normalized Poincar\'e bundle on $X \times \rM(r,A^{r} \otimes L^*) \cong X \times \rM(r, L)$. The restriction of  $\cE^* \otimes p^* A\otimes q^* \Theta$ to $x \times \rM(r,L)$ is isomorphic to $\cE_x^* \otimes \Theta$. From Theorem \ref{thm:nef}, we see that $\cE_{x}^{*}\otimes \Theta$ is strictly nef. 
\end{proof}

From now on, we prove the nefness of $\cE_{x}$. By definition, we need to show that $\cO_{\PP (\cE_{x})}(1)$ is nef. Observe that $\PP (\cE_{x}) \cong \rM(r, L, r-1, \epsilon)$ for some very small $\epsilon > 0$ (Example \ref{ex:smallweight}).  

We explicitly analyze the first wall-crossing of the moduli space $\rM(r, L, r-1, \epsilon)$ by increasing $\epsilon \to 1$. Recall that $\ell$ is a positive integer such that $\ell d \equiv 1 \;\mbox{mod} \;r$ and $0 < \ell < r$. 

\begin{lemma}\label{lem:1stwallcrossing}
Let $a$ be the smallest parabolic weight on a wall. Then $a = 1/\ell$. Furthermore, a maximal destabilizing subbundle has rank $k\ell$ and degree $ke$ for some $k \in \ZZ$ and an integer $e$ satisfying $\ell d - re = 1$. 
\end{lemma}

\begin{proof}
Let $\Delta(s, e, n)$ be a wall. Note that $n$ is either $s$ or $s-1$. By Equation \eqref{eqn:wallduality}, by exchanging $s$ by $r-s$ if necessary, we may assume that $n = s$. Then from $(e+sa)/s = (d+(r-1)a)/r$, we have $a = (sd - re)/s$. Since $(r, d) = 1$, we can find a unique positive $0 < s < r$ and $e \in \ZZ$ such that $sd - re = 1$, which is $\ell$. 

We claim that $a = (\ell d - re)/\ell = 1/\ell$ provides the first wall. Suppose that there is another wall $a' = (s'd - re')/s'$. Then $s'd - re' = k > 1$. So $s'd \equiv k \;\mathrm{mod}\; r$. On the other hand, $k\ell d \equiv k \;\mathrm{mod}\; r$. So if $k\ell < r$, then $s' = k\ell$ and $e' = ke$. Then $a' = k/k\ell = 1/\ell = a$. If $k\ell \ge r$, there is a unique positive integer $t$ such that $0 < s' = k\ell - tr < r$. Then $a = k/s' = k/(k\ell-tr) > k/k\ell = 1/\ell$. 

The above numerical computation tells us that $\Delta(\ell, e, \ell) = \Delta(s', e', s')$ only if $(s', e')=(k\ell, ke)$. So we obtain the last assertion. 
\end{proof}

We have the following diagram:
\[
	\xymatrix{&\PP(\cE_{x}) = \rM(r, L, r-1, \epsilon) \ar[ld]_{\pi} \ar[rd]^{\pi_{-}}\\
	\rM(r, L) && \rM(r, L, r-1, 1/\ell)}
\]
The first map $\pi$ is a projective bundle and $\pi_{-}$ is a small contraction by Corollary \ref{cor:codimcenter}. And $\rho(\PP(\cE_{x})) = \rho(\rM(r, L)) + 1 = 2$. Since $\pi_{-}$ is a small contraction, $1 \le \rho(\rM(r, L, r-1, 1/\ell)) < \rho(\rM(r, L, r-1, \epsilon)) = 2$, so $\rho(\rM(r, L, r-1, 1/\ell)) = 1$. Let $A$ be an ample generator of $\mathrm{Pic}(\rM(r, L, r-1, 1/\ell))$. Then $\pi^{*}\Theta$ and $\pi_{-}^{*}A$ generates $\rN^{1}(\PP(\cE_{x}))_{\RR}$. 

\begin{definition}\label{def:twocurves}
Fix a \emph{general} point $((E^{-}, V^{-}), (E^{+}, V^{+}))$ in the component $\rM(\ell, e, \ell, 1/\ell) \times_{\mathrm{Pic}(X)}\rM(r-\ell, d-e, r-\ell-1, 1/\ell)$ of the wall-crossing center in $\rM(r, L, r-1, 1/\ell)$. The fiber $\pi_{-}^{-1}(((E^{-}, V^{-}), (E^{+}, V^{+})))$ is a projective space $\PP \Ext^{1}((E^{+}, V^{+}), (E^{-}, V^{-}))$. Let $C$ be a line class in it. 
\end{definition}

\begin{lemma}\label{lem:intersection2}
The intersection number $\cO_{\PP (\cE_{x})}(1) \cdot C$ is zero. 
\end{lemma}

\begin{proof}
The image $\pi(\PP \Ext^{1}((E^{+}, V^{+}), (E^{-}, V^{-}))) = \PP \Ext^{1}(E^{+}, E^{-}) =: \PP$ parametrizes isomorphism classes of extensions, and there is an exact sequence over $X \times \PP$
\[
	0 \to p^{*}E^{-} \otimes q^{*}\cO_{\PP}(1) \to E \otimes q^{*}\cO_{\PP}(m) \to p^{*}E^{+} \to 0
\]
(\cite[Lemma 2.3]{Ram73}, \cite[Example 2.1.12]{HL10}). Here $p : X \times \PP \to X$ and $q : X \times \PP \to \PP$ are two projections. If we restrict the exact sequence to $x \times C \cong x \times \PP^{1} \subset X \times \PP$, we have 
\[
	0 \to E^{-}_{x} \otimes \cO_{\PP^{1}}(1) \to E_{x} \otimes \cO_{\PP^{1}}(m) \to E^{+}_{x} \to 0.
\]
Since $\cE_{x}$ (and hence its restriction $E_{x}$) is normalized as $c_{1}(\cE_{x}) = \Theta^{\ell}$ where $0 < \ell < r$, and $E^{-}_{x}$ and $E^{+}_{x}$ are constant, $\ell = c_{1}(E^{-}_{x} \otimes \cO_{\PP^{1}}(1)) = c_{1}(E_{x} \otimes \cO_{\PP^{1}}(m)) = \ell + rm$. Thus, we have $m = 0$. Then $E_{x}|_{\pi(C)}$ fits in $0 \to \cO_{\PP^{1}}(1)^{\ell} \to E_{x} \to \cO_{\PP^{1}}^{r-\ell} \to 0$. By a cohomology computation, we can show that this is a split extension. Therefore $\pi^{-1}(\pi(C)) = \PP(\cO_{\PP^{1}}(1)^{\ell} \oplus \cO_{\PP^{1}}^{r-\ell})$. The parabolic flag in $E_{x}$ is determined by that of $E^{+}_{x}$ and it is fixed over $C$. This implies that $C \cong \PP (\cO_{\PP^{1}}) \hookrightarrow \PP(\cO_{\PP^{1}}(1)^{\ell} \oplus \cO_{\PP^{1}}^{r-\ell})$. Therefore $\cO_{\PP(\cE_{x})}(1)|_{C} = \cO_{\PP (\cO_{\PP^{1}})}(1) = \cO_{\PP^{1}}$ and $\cO_{\PP (\cE_{x})}(1) \cdot C = 0$. 
\end{proof}

\begin{proof}[Proof of Theorem \ref{thm:nef}]
From $\rho(\PP (\cE_{x})) = 2$, $\pi_{-}^{*}A \cdot C = 0$, and Lemma \ref{lem:intersection2}, we can conclude that $\cO_{\PP (\cE_{x})}(1)$ and $\pi_{-}^{*}A$ are proportional. $\cO_{\PP(\cE_{x})}(1)$ is a positive multiple of $\pi_{-}^{*}A$ because it intersects with the line class in a fiber of $\pi : \PP (\cE_{x}) \to \rM(r, L)$ positively. Therefore $\cO_{\PP(\cE_{x})}(1)$ is semi-ample, and it is nef. By definition, $\cE_{x}$ is nef. $\cE_{x}$ is strictly nef because $\cO_{\PP(\cE_{x})}(1)$ is not ample.
\end{proof}

We immediately obtain the nef cones of $\PP(\cE_{x})$ and $\PP(\cE_{x}^{*})$. The bigness in the statements follows from Lemma \ref{lem:pullbackofample} and Corollary \ref{cor:pullbackofample}. 

\begin{corollary}\label{cor:nefcone}
The nef cone of $\PP(\cE_{x}) = \rM(r, L, r-1, \epsilon)$ is generated by $\pi^{*}\Theta$ and $\cO_{\PP(\cE_{x})}(1)$. If $d \ne 1$, $\cO_{\PP(\cE_{x})}(1)$ is big. 
\end{corollary}

\begin{corollary}\label{cor:nefconedual}
The nef cone of $\PP(\cE_{x}^{*}) = \rM(r, L, 1, \epsilon)$ is generated by $\pi^{*}\Theta$ and $\cO_{\PP(\cE_{x}^{*})}(1)\otimes \pi^{*}\Theta$. If $d \ne r-1$, $\cO_{\PP(\cE_{x}^{*})}(1) \otimes \pi^{*}\Theta$ is big. 
\end{corollary}

\begin{remark}\label{rmk:d=1}
It is worth to point out a difference in $d = 1$ case. The numerical computation in Lemma \ref{lem:1stwallcrossing} is still valid. But in this case, from $\ell d \equiv 1 \;\mathrm{mod}\; r$, we have $\ell = 1$ and thus, $a = 1$. Therefore, the first wall-crossing is precisely the fibration $\rM(r, L, r-1, \epsilon) \to \rM(r, L(-x))$ in Proposition \ref{prop:generalizedHeckemodification}, that is, a contraction in the Hecke correspondence. 
\end{remark}

\section{Vanishing of cohomology and embedding of the derived category}\label{sec:derivedcategory}

The aim of this section is to prove Theorem \ref{thm:mainthm}. 

\subsection{Bondal-Orlov criterion}

Let $\cE$ be the normalized Poincar\'e bundle over $X \times \rM(r, L)$. Let $p : X \times \rM(r, L) \to X$, $q : X \times \rM(r, L) \to \rM(r, L)$ be two projections. Consider the Fourier-Mukai transform 
\begin{eqnarray*}
	\Phi_{\cE} : \rD^{b}(X) &\to& \rD^{b}(\rM(r, L))\\
	F^{\bullet} & \mapsto & Rq_{*}(\cE \otimes^L Lp^{*} F^{\bullet}).
\end{eqnarray*}

The Bondal-Orlov criterion (\cite[Theorem 1.1]{BO95}) provides the necessary and sufficient condition for the fully-faithfulness of a Fourier-Mukai transform between two smooth algebraic varieties. The next theorem is a version applied to $\Phi_{\cE}$.

\begin{theorem}[Bondal-Orlov criterion]\label{thm:vanishing}
For each $x \in X$, let $\cE_{x}$ be the restriction of the normalized Poincar\'e bundle on $\rM(r, L)$. Then $\Phi_{\cE} : \rD^{b}(X) \to \rD^{b}(\rM(r, L))$ is fully faithful if and only if the following conditions hold:
\begin{enumerate}
\item $\rH^{0}(\rM(r, L), \cE_{x} \otimes \cE_{x}^{*}) \cong \CC$. 
\item $\rH^{i}(\rM(r, L), \cE_{x} \otimes \cE_{x}^{*}) = 0$ for $i \ge 2$. 
\item $\rH^{i}(\rM(r, L), \cE_{x} \otimes \cE_{y}^{*}) = 0$ for all $x \ne y$ and all $i$. 
\end{enumerate}
\end{theorem}

The main result of this section is to show the vanishing of cohomologies when $g \ge r+3$. Then Theorem \ref{thm:mainthm} follows immediately. 

\begin{proof}[Proof of Theorem \ref{thm:mainthm}]
Items (1) and (2) in Theorem \ref{thm:vanishing} are already proved in \cite[Section 3]{BM19} by extending the work of Narasimhan and Ramanan in \cite{NR75}. Item (3) is obtained by combining Corollary \ref{cor:vanishingforEx} and Proposition \ref{prop:highervanishing}. When $(r-1)(g-1) \ge r^{2}$, or equivalently, $g \ge r+3$, all $i$'s are covered by the above two statements. 
\end{proof}

\begin{remark}
For $d = 1$, Belmans and Mukhopadhyay proved the theorem for $g \ge r+3$ (\cite[Theorem 3]{BM19}). 
\end{remark}

\begin{remark}
We expect that the genus bound in Theorem \ref{thm:mainthm} is not essential. It would be a very interesting task to prove the statement for every rank and genus.
\end{remark}

\subsection{Vanishing of cohomology}\label{ssec:cohomologyvanishing}

From now on, we investigate cohomology of line bundles on $\rM(r, L, \bm, \ba)$ where there are two parabolic points $\bp = (x, y)$ and $\bm = (r-1, 1)$. When the parabolic weight $\ba$ is sufficiently small, $\rM(r, L, \bm, \ba) \cong \PP(\cE_{x}) \times_{\rM(r, L)}\PP(\cE_{y}^{*})$ by Example \ref{ex:smallweight}. By investigating the wall-crossing, we show the vanishing of some cohomology groups on $\rM(r, L, \bm, \ba)$. 

The following lemma is obtained by essentially the same computation with \cite[Proposition 3.1]{Nar17}. 

\begin{lemma}\label{lem:DetE}
Let $\cE$ be the normalized Poincar\'e bundle on $X \times \rM(r, L)$. Then 
\[
	\Det(\cE^{*}) := \det(Rq_{*}(\cE^{*}))^{-1} \cong \Theta^{\ell(1-g) - e}.
\]
\end{lemma}

\begin{proof}
For the notational simplicity, let $\rM := \rM(r, L)$ and $\rM' := \rM(r, L^{*})$. Then there is an isomorphism $\psi : \rM \to \rM'$. Since the isomorphism maps the unique ample generator $\Theta_{\rM'}$ to $\Theta_{\rM}$, by \cite[Proposition 2.1]{Nar17}, 
\[
	\Theta_{\rM} = \psi^{*}(\Theta_{\rM'}) = (\Det(\cE^{*}))^{r} \otimes (\det(\cE^{*}|_{\{x\} \times \rM}))^{-d+r(1-g)}
	= \Det(\cE^{*})^{r} \otimes \Theta_{\rM}^{-\ell(-d+r(1-g))}.
\]
Thus, $\Det(\cE^{*}) = \Theta_{\rM}^{\frac{1 + \ell (-d + r(1-g))}{r}} = \Theta_{\rM}^{-e + \ell(1-g)}$.
\end{proof}

Once we fix the parabolic points and the multiplicity, $\rM(r, L, \bm, \ba)$ are all birational, and for any general $\ba$ and $\ba'$, $\rM(r, L, \bm, \ba)$ and $\rM(r, L, \bm, \ba')$ are connected by finitely many flips. In particular, their Picard groups are identified. For a notational simplicity, we will suppress all pull-backs (by flips and regular contractions) in our notation. For instance, when there is only one parabolic point $x$, there are two rational contractions $\pi : \rM(r, L, r-1, \epsilon) \to \rM(r, L)$ and $\pi_{1} : \rM(r, L, r-1, \epsilon) \dashrightarrow \rM(r, L, r-1, 1-\epsilon) \to \rM(r, L(-x))$. If there is no chance of confusion, we use $A \otimes B$ instead of $\pi^{*}A \otimes \pi_{1}^{*}B$. We denote $\cO_{\PP (\cE_{x})}(a)$ by $\cO(a)$. We also set $\cO(a, b) := p_{1}^{*}\cO_{\PP(\cE_{x})}(a) \otimes p_{2}^{*}\cO_{\PP(\cE_{y}^{*})}(b)$ where $p_{1} : \PP(\cE_{x})\times_{\rM(r, L)}\PP(\cE_{y}^{*}) \to \PP(\cE_{x})$ and $p_{2} : \PP(\cE_{x})\times_{\rM(r, L)}\PP(\cE_{y}^{*}) \to \PP(\cE_{y}^{*})$.

\begin{lemma}\label{lem:pullbackofample}
Let $k = (r, d-1)$. On $\rM(r, L, r-1, a)$, 
\[
	\Theta_{\rM(r, L(-x))}^{k} = \cO_{\PP(\cE_{x})}(r) \otimes \Theta_{\rM(r, L)}^{1-\ell}.
\]
\end{lemma}

\begin{proof}
The proof is a careful refinement of that of \cite[Proposition 3.3]{Nar17}. We may assume that $a$ is sufficiently small, so $\rM(r, L, r-1, a) \cong \PP(\cE_{x})$. 

Let $p : X \times \PP(\cE_{x}) \to X$ and $q : X \times \PP(\cE_{x}) \to \PP(\cE_{x})$ be two projections and $\pi : X \times \PP(\cE_{x}) \to X \times \rM(r, L)$. Let $i_{x} : \PP (\cE_{x}) \cong x \times \PP(\cE_{x}) \hookrightarrow X \times \PP(\cE_{x})$. Recall that there are two exact sequences that appear on the construction of the Hecke correspondence:
\[
	0 \to H(\cE) \to \pi^{\#}(\cE) \to p^{*}\cO_{x} \otimes q^{*}\cO_{\PP (\cE_{x})}(1) \to 0
\]
and 
\begin{equation}\label{eqn:sesforpiE}
	0 \to \pi^{\#}(\cE^{*}) \to K(\cE) \to i_{x *}(\cO_{\PP(\cE_{x})}(-1)\otimes T_{x}) \to 0.
\end{equation}
Here $\pi^{\#}\cE$ is the pull-back of $\cE$ to $X \times \PP(\cE_{x})$ and $T_{x}$ is the tangent space of $X$ at $x$. 

By \cite[Proposition 2.1]{Nar17}, 
\[
	\Theta_{\rM(r, L(-x))}^{k} = \Theta_{\rM(r, L^{*}(x))}^{k} = \mathrm{Det}(K(\cE))^{r} \otimes (\det K(\cE)|_{z \times \PP(\cE_{x})})^{1-d+r(1-g)}
\]
for any $z \in X$. From \eqref{eqn:sesforpiE}, we have $\mathrm{Det}(\pi^{\#}(\cE^{*})) \otimes \cO_{\PP(\cE_{x})}(1) = \mathrm{Det}(K(\cE))$. Since $\mathrm{Det}(\pi^{\#}(\cE^{*})) = \pi^{\#}\mathrm{Det}(\cE)$ and $\pi^{\#}(\cE^{*})|_{z \times \PP(\cE_{x})} \cong K(\cE)|_{z \times \PP(\cE_{x})}$ for any $z \ne x$, 
\begin{equation}
\begin{split}
	&\mathrm{Det}(K(\cE))^{r} \otimes (\det K(\cE)|_{z \times \PP(\cE_{x})})^{1-d+r(1-g))}\\ 
	&= \mathrm{Det}(K(\cE))^{r}\otimes (\det \pi^{\#}(\cE^{*})|_{z \times \PP(\cE_{x})})^{1-d+r(1-g)}
	= \mathrm{Det}(K(\cE))^{r} \otimes \Theta_{\rM(r, L)}^{-\ell(1-d+r(1-g))}\\
	&= \pi^{\#}(\mathrm{Det}(\cE^{*}))^{r} \otimes \cO_{\PP(\cE_{x})}(r) \otimes \Theta_{\rM(r, L)}^{-\ell(1-d+r(1-g))}\\
	&= \Theta_{\rM(r, L)}^{r\ell(1-g)-re} \otimes \cO_{\PP(\cE_{x})}(r) \otimes \Theta_{\rM(r, L)}^{-\ell(1-d+r(1-g))}
	= \cO_{\PP (\cE_{x})}(r) \otimes \Theta_{\rM(r, L)}^{1-\ell}.
\end{split}
\end{equation}
The second and the fourth equalities follow from the normalization of $\cE$ and Lemma \ref{lem:DetE}, respectively. 
\end{proof}

\begin{corollary}\label{cor:pullbackofample}
Let $k = (r, d-(r-1))$. Then 
\[
	\Theta_{\rM(r, L(-(r-1)y))}^{k} = \cO_{\PP(\cE_{y}^{*})}(r) \otimes \Theta_{\rM(r, L)}^{1+\ell}.
\]
\end{corollary}

\begin{proof}
Within the identification $\rM(r, L) \cong \rM(r, L^{*})$, the normalized Poincar\'e bundle over $\rM(r, L^{*})$ is $\cE^{*}\otimes \Theta_{\rM(r, L)}$, and $c_{1}(\cE^{*}\otimes \Theta_{\rM(r, L)}) = c_{1}(\Theta_{\rM(r, L)}^{r-\ell})$. So $\rM(r, L, 1, \epsilon) \cong \rM(r, L^{*}, r-1, \epsilon) \cong \PP(\cE_{y}^{*} \otimes \Theta_{\rM(r, L)})$. When $a \to 1$, we obtain a contraction $\rM(r, L^{*}, r-1, a) \to \rM(r, L^{*}(-y)) \cong \rM(r, L(y)) \cong \rM(r, L(-(r-1)y))$. By Lemma \ref{lem:pullbackofample}, 
\[
	\Theta_{\rM(r, L(-(r-1)y))}^{k} = \Theta_{\rM(r, L^{*})}^{1-(r-\ell)} \otimes\cO_{\PP (\cE_{y}^{*} \otimes \Theta)}(r) = \Theta_{\rM(r, L)}^{1 - (r-\ell)} \otimes \cO_{\PP(\cE_{y}^{*})}(r) \otimes \Theta_{\rM(r, L)}^{r} = \cO_{\PP(\cE_{y}^{*})}(r) \otimes \Theta_{\rM(r, L)}^{1+\ell}.
\]
\end{proof}

From now on, $\bp = (x, y)$ and $\bm = (r-1, 1)$. The line bundle $\Theta$ is the pull-back of $\Theta$ by $\rM(r, L, \bm, \ba) \dashrightarrow \rM(r, L)$. 

\begin{lemma}\label{lem:canonicaldivisor}
For a general weight $\ba$, the dualizing bundle of $\rM(r, L, \bm, \ba)$ is
\[
	\omega = \cO(-r, -r) \otimes \Theta^{-2}.
\]
\end{lemma}

\begin{proof}
We may assume that $\ba$ is sufficiently small and $\rM(r, L, \bm, \ba) \cong \PP(\cE_{x}) \times_{\rM(r, L)}\PP(\cE_{y}^{*})$. It follows from the relative Euler sequence applied for $\PP(\cE_{x}) \to \rM(r, L)$ and $\PP(\cE_{x}) \times_{\rM(r, L)}\PP(\cE_{y}^{*}) \to \PP(\cE_{x})$. 
\end{proof}

\begin{proposition}\label{prop:effectivecone}
Let $\ba$ be a general weight. The effective cone of $\rM(r, L, \bm, \ba)$ is generated by four extremal rays
\[
	\Theta, \cO(r, 0) \otimes \Theta^{1-\ell}, \cO(0, r) \otimes \Theta^{1+\ell}, \cO(r, r) \otimes \Theta.
\]
\end{proposition}
\begin{proof}
By Proposition \ref{prop:weightanddivisor}, it is sufficient to find four divisors associated to four extremal parabolic weights. When $\ba = (a_{x}, a_{y}) = (0, 0)$, the associated rational contraction is $\rM(r, L)$ and the associated divisor is a scalar multiple of $\Theta$. When $\ba = (1/\ell, 0)$, by Section \ref{sec:nef}, the associated divisor is a multiple of $\cO(1, 0)$. When $\ba = (1, 0)$, the associated rational contraction is $\rM(r, L(-x))$ and the associated divisor is a scalar multiple of $\cO(r, 0) \otimes \Theta^{1-\ell}$ by Lemma \ref{lem:pullbackofample}. For $\ba = (0, 1/(r-\ell))$, we have a multiple of $\cO(0, 1) \otimes \Theta$. Finally, for $\ba = (0, 1)$, a multiple of $\cO(0, r) \otimes \Theta^{1+\ell}$ is associated. 

By an elementary computation, for each point $\ba \in [0, 1]^{2}$, the associated divisor can be written as a multiple of $\Theta \otimes (\cO(r, 0) \otimes \Theta^{-\ell})^{a_{x}} \otimes (\cO(0, r) \otimes \Theta^{\ell})^{a_{y}}$. Thus, the last extremal ray, which is associated to $\ba = (1, 1)$, is $\cO(r, r) \otimes \Theta$. 
\end{proof}

\begin{corollary}\label{cor:positivity}
For some general parabolic weight $\ba$, $\cO(r+1, r+1) \otimes \Theta^{2}$ is nef and big on $\rM(r, L, \bm, \ba)$.
\end{corollary}

\begin{proof}
The statement is immediate from the fact that the line bundle is on the interior of the effective cone, and there is no divisorial contraction in the wall-crossing (Corollary \ref{cor:codimcenter}).
\end{proof}

Recall that a normal $\QQ$-factorial variety is of Fano type if there is an effective $\QQ$-divisor $\Delta$ such that $-(K + \Delta)$ is ample and $(X, \Delta)$ is a klt pair. 

\begin{corollary}\label{cor:Fanotype}
For any general $\ba$, the moduli space $\rM(r, L, \bm, \ba)$ is of Fano type.
\end{corollary}
\begin{proof}
It also follows from the fact that $\omega^{*} = \cO(-K) = \cO(r,r) \otimes \Theta^{2}$ is on the interior of the effective cone. If we pick a general weight $\ba'$ such that the nef cone of $\rM(r, L, \bm, \ba')$ includes $-K$, then $-K$ is nef and big, so $\rM(r, L, \bm, \ba')$ is a smooth weakly Fano variety, hence of Fano type. For a general $\ba \in (0, 1)^{2}$, $\rM(r, L, \bm, \ba)$ is obtained from $\rM(r, L, \bm, \ba')$ by applying finitely many flips. Therefore it is of Fano type by \cite[Theorem 1.1]{GOST15}. 
\end{proof}

\begin{corollary}\label{cor:vanishing}
For $0 < i < (r-1)(g-1)$, $\rH^{i}(\rM(r, L, \bm, (\epsilon, \epsilon)), \cO(1, 1)) = 0$. 
\end{corollary}

\begin{proof}
Since $\cO(1, 1) = \cO(r+1, r+1)\otimes \Theta^{2} \otimes \omega$, for $\ba$ in Corollary \ref{cor:positivity}, $\rH^{i}(\rM(r, L, \bm, \ba), \cO(1, 1)) = 0$ for $i > 0$ by the Kawamata-Viehweg vanishing theorem. Since $\rM(r, L, \bm, \ba)$ and $\rM(r, L, \bm, (\epsilon, \epsilon))$ are connected by finitely many flips with the flipping centers of codimension $\ge (r-1)(g-1) + 1$ (Corollary \ref{cor:codimcenter}), $\rH^{i}(\rM(r, L, \bm, (\epsilon, \epsilon)), \cO(1, 1)) = \rH^{i}(\rM(r, L, \bm, \ba), \cO(1, 1))$ for $i < (r-1)(g-1)$ (\cite[III. Lemma 3.1]{Gro05}, \cite[Theorem 3.8]{Har67}). 
\end{proof}

\begin{corollary}\label{cor:vanishingforEx}
For $i < (r-1)(g-1)$, $\rH^{i}(\rM(r, L), \cE_{x}\otimes \cE_{y}^{*}) = 0$. 
\end{corollary}

\begin{proof}
For $0 < i < (r-1)(g-1)$, it follows from Corollary \ref{cor:vanishing} and the Leray spectral sequence. For $i = 0$, it follows from the stability of $\cE_{x}$, $\cE_{y}$ (\cite[Proposition 2.1]{LN05}), and the fact that $\cE_{x} \ne \cE_{y}$ if $x \ne y$ (\cite[Theorem]{LN05}). 
\end{proof}

Finally, the vanishing of the higher cohomology groups are obtained by Le Potier vanishing theorem (\cite[Theorem 7.3.5]{Laz04}). 

\begin{proposition}\label{prop:highervanishing}
For $i \ge r^{2}$, $\rH^{i}(\rM(r, L), \cE_{x}\otimes \cE_{y}^{*}) = 0$. 
\end{proposition}

\begin{proof}
Note that
\[
\begin{split}
	\rH^{i}(\rM(r, L), \cE_{x} \otimes \cE_{y}^{*}) &= \rH^{i}(\rM(r, L), \omega_{\rM(r, L)} \otimes \cE_{x} \otimes \cE_{y}^{*} \otimes \Theta^{2})\\
	 &= \rH^{i}(\rM(r, L), \omega_{\rM(r, L)} \otimes \cE_{x} \otimes (\cE_{y}^{*}  \otimes \Theta) \otimes \Theta).
\end{split}
\]
The bundle $\cE_{x} \otimes (\cE_{y}^{*}  \otimes \Theta) \otimes \Theta$ is ample because it is a tensor product of nef and ample bundles (\cite[Theorem 6.2.12. (iv)]{Laz04}). Thus, Le Potier vanishing theorem implies the desired vanishing. 
\end{proof}

\section{ACM bundles on $\rM(r,L)$}\label{sec:ACM}

We now turn to our second application (Theorem \ref{thm:ACMintro}) of the nefness of $\cE_{x}$. 

Let $V$ be an $n$-dimensional projective variety with an ample line bundle $A$. We recall the definition of ACM bundles. 

\begin{definition}\label{def:ACM}
A vector bundle $\cE$ on $V$ is an \emph{ACM bundle} with respect to $A$ if $\rH^i(V,\cE \otimes A^{j})=0$ for every $1 \leq i \leq n-1$ and $j \in \ZZ$. An ACM bundle $\cE$ is \emph{Ulrich} if $\rH^0(V, \cE \otimes A^{-1})=0$ and $\rH^0(V,\cE)=\rk \cE \cdot \deg V = \rk \cE \cdot (A)^{n}$.
\end{definition}

For a smooth Fano variety of Picard rank one, it is straightforward to verify that every line bundle is ACM. It is also clear that if $\cE$ is ACM with respect to $A$, then $\cE \otimes A^{k}$ is ACM with respect to $A$ for all $k \in \ZZ$. But finding a non-trivial example of an ACM bundle is not an easy task for higher dimensional varieties. In this section, we show that $\cE_{x}$ is ACM if $g \ge 3$.

\begin{remark}\label{rmk:veryample}
Many authors assume that $A$ to be very ample when they consider ACM bundles. Because the Picard number of $\rM(r, L)$ is one, Theorem \ref{thm:ACMintro} implies that $\cE_x$ is ACM for every very ample line bundle. On $\rM(r, L)$, $\Theta^{k}$ is known to be very ample when $k \ge r^{2}+r$ (\cite[Theorem A]{EP04}), but an optimal $k$ for the very ampleness is unknown.
\end{remark}

Fix a point $x \in X$ and consider bundle morphisms $\PP (\cE_{x}) \to \rM(r, L)$ and $\PP (\cE_{x}^{*}) \to \rM(r, L)$. From the relative Euler sequence, we have 
\begin{equation}\label{eqn:dualizingsheaves}
	\omega_{\PP(\cE_{x})} \cong \cO(-r) \otimes \Theta^{\ell - 2}, \quad \omega_{\PP (\cE_{x}^{*})} \cong \cO(-r) \otimes \Theta^{-\ell-2}.
\end{equation}

Here we use the notational convention in Section \ref{ssec:cohomologyvanishing}. We set $n := \dim \rM(r, L) = (r^{2}-1)(g-1)$ and assume that $g \ge 2$. We state three vanishing results coming from different sources. 

\begin{lemma}\label{lem:Kodairavanishing}
We have $\rH^i(\rM(r, L),\cE_x \otimes \Theta^{j}) = 0$ for $i \geq 1$, $j \geq \ell-1$.
\end{lemma}

\begin{proof}
By \eqref{eqn:dualizingsheaves}, $	\cO(1) \otimes \Theta^{j} \cong \omega_{\PP(\cE_x)} \otimes \cO(r+1) \otimes \Theta^{2-\ell+j}$. By  Kodaira  vanishing theorem and Corollary \ref{cor:nefcone}, we have 
\begin{equation}\label{eqn:Leraysequence}
	\rH^i(\rM(r, L),\cE_x \otimes \Theta^{j}) \cong \rH^i(\PP(\cE_x), \cO(1) \otimes \Theta^{j}) \cong \rH^i(\PP(\cE_x), \omega_{\PP (\cE_x)} \otimes \cO(r+1) \otimes \Theta^{2-\ell+j}) = 0
\end{equation}
for $i \geq 1$, $j \geq \ell-1$. 
\end{proof}

\begin{lemma}\label{lem:LePotiervanishing}
We have $\rH^{i}(\rM(r, L), \cE_{x} \otimes \Theta^{j}) = 0$ for $i \ge r$, $j \ge -1$.
\end{lemma}

\begin{proof}
Since $\cE_{x} \otimes \Theta$ is ample (\cite[Theorem 6.2.12. (iv)]{Laz04}), Le Potier vanishing theorem (\cite[Theorem 7.3.5]{Laz04}) immediately implies that 
\[
	\rH^{i}(\rM(r, L), \cE_{x} \otimes \Theta^{j}) = \rH^{i}(\rM(r, L), \omega_{\rM(r, L)} \otimes \cE_{x} \otimes \Theta^{j+2}) = 0
\]
for $i \ge r$ and $j \ge -1$. 
\end{proof}

\begin{lemma}\label{lem:wallcrossing}
We have $\rH^{i}(\rM(r, L), \cE_{x} \otimes \Theta^{j}) = 0$ for $1 \le i \le (r-1)(g-1) -1$, $j > -1+(1-\ell)/r$.
\end{lemma}

\begin{proof}
By \eqref{eqn:Leraysequence}, it is sufficient to show that $\rH^{i}(\PP(\cE_{x}), \omega \otimes \cO(r+1) \otimes \Theta^{2-\ell +j}) = 0$. Since $\PP(\cE_{x})$ and $\rM(r, L, r-1, a)$ for a general parabolic weight $a$ is connected by finitely many flips with wall-crossing centers of codimension $\ge (r-1)(g-1) + 1$, for $0 < i < (r-1)(g-1)$, it is sufficient to show the vanishing for some $a$ (\cite[III. Lemma 3.1]{Gro05}, \cite[Theorem 3.8]{Har67}). Proposition \ref{prop:effectivecone} implies that the effective cone of $\PP(\cE_{x}) = \rM(r, L, r-1, \epsilon)$, which is identified to the boundary of $\mathrm{Eff}(\rM(r, L, \bm, (\epsilon, \epsilon)))$ given by $a_{y} = 0$, is generated by $\Theta$ and $\cO(r) \otimes \Theta^{1-\ell}$. Thus, for a bundle $F = \cO(a) \otimes \Theta^{b}$, if $a > 0$ and $b/a > (1-\ell)/r$, then $F$ is big. Thus, for some general parabolic weight $a$, by Kawamata-Viehweg vanishing theorem, 
\[
	\rH^{i}(\rM(r, L, r-1, a), \omega \otimes \cO(r+1) \otimes \Theta^{2-\ell+j}) = 0
\]
for $i \ge 1$ and $(2-\ell +j)/(r+1) > (1-\ell)/r$, or equivalently, $j > -1 + (1-\ell)/r$. 
\end{proof}

\begin{proof}[Proof of Theorem \ref{thm:ACMintro}]
We divide the computation into several steps.  

\textsf{Step 1.} It is sufficient to show that $\rH^{i}(\rM(r, L), \cE_{x} \otimes \Theta^{j}) = 0$ for $1 \le i \le n-1$ and $j \ge -1$. 

By Serre duality, $\rH^{i}(\rM(r, L), \cE_{x} \otimes \Theta^{j}) \cong \rH^{n-i}(\rM(r, L), \cE_{x}^{*} \otimes \Theta \otimes \Theta^{-j-3})$. Since $\cE_{x}^{*}\otimes \Theta$ is the restriction of the normalized Poincar\'e bundle over $\rM(r, L^{*}(r)) \cong \rM(r, L)$, the vanishing of $\cE_{x} \otimes \Theta^{j}$ for $j \le -2$ follows from the vanishing of $\cE_{x}^{*}\otimes \Theta \otimes \Theta^{j}$ for $j \ge -1$.

\textsf{Step 2.} $\ell \ne 1$. 

It is straightforward to check that, if $g \ge 3$, then the vanishing results in Lemmas \ref{lem:Kodairavanishing}, \ref{lem:LePotiervanishing}, and \ref{lem:wallcrossing} imply that $\cE_{x}$ is ACM. 

\textsf{Step 3.} $\ell = 1$.

The above three lemmas cover all cohomology groups except $1 \le i \le r-1$, $j = -1$. For the $\ell = 1$ case, there is a contraction map $\pi_{1} : \PP(\cE_{x}) = \rM(r, L, r-1, \epsilon) \to \rM(r, L(-x))$ (Remark \ref{rmk:d=1}). Then by \cite[Lemma 13]{BM19}, 
\[
\begin{split}
	\rH^{i}(\rM(r, L), \cE_{x} \otimes \Theta^{-1}) &\cong \rH^{i}(\PP(\cE_{x}), \cO(1) \otimes \Theta^{-1}) = \rH^{i}(\PP(\cE_{x}), \omega_{\PP(\cE_{x})} \otimes \cO(r+1))\\
	&= \rH^{i}(\PP(\cE_{x}), \omega_{\PP(\cE_{x})} \otimes \pi_{1}^{*}\Theta_{\rM(r, L(-x))}^{r+1}).
\end{split} 
\]

By Koll\'ar's vanishing theorem (\cite[Theorem 2.1]{Kol86}), $R^{i}\pi_{1 *}\omega_{\PP(\cE_{x})}$ is torsion free for all $i$ and 
\[
	\rH^{k}(\rM(r, L(-x)), R^{i}\pi_{1 *}\omega_{\PP(\cE_{x})} \otimes \Theta_{\rM(r, L(-x))}^{r+1}) = 0
\]
for all $k > 0$. Since the Leray spectral sequence degenerates, $\rH^{0}(\rM(r, L(-x)), R^{i}\pi_{1 *}\omega_{\PP(\cE_{x})} \otimes \Theta_{\rM(r, L(-x))}^{r+1}) \cong \rH^{i}(\PP(\cE_{x}), \omega_{\PP(\cE_{x})} \otimes \pi_{1}^{*}\Theta_{\rM(r, L(-x))}^{r+1})$. On the other hand, over the stable locus $\rM(r, L(-x))^{s}$, $\pi_{1}$ is a $\PP^{r-1}$-fibration. Checking a general fiber, we can show that $R^{i}\pi_{1 *}\omega_{\PP(\cE_{x})} = 0$ for $i \ne r - 1$. Thus, we obtain the desired vanishing for $1 \le i \le r-2$. 

For $i = r-1$, since $R^{r-1} \pi_{1 *}\omega_{\PP(\cE_{x})}$ is a torsion free sheaf, we have an injective morphism $R^{r-1} \pi_{1 *}\omega_{\PP(\cE_{x})} \hookrightarrow (R^{r-1} \pi_{1 *}\omega_{\PP(\cE_{x})})^{\vee\vee}$. These two are isomorphic to $\omega_{\rM(r, L(-x))}$ over an open subset of codimension $\ge 2$ (\cite[Exercise III.8.4]{Har77}) and the latter is reflexive. Since $\rM(r, L(-x))$ is locally factorial (\cite[Theorem A]{DN89}), $(R^{r-1} \pi_{1 *}\omega_{\PP(\cE_{x})})^{\vee\vee} \cong \omega_{\rM(r, L(-x))} \cong \Theta_{\rM(r, L(-x))}^{-2r}$ (\cite[Theorem F]{DN89}). Now we have 
\[
\begin{split}
	\rH^{0}(\rM(r, L(-x)), R^{r-1}\pi_{1 *}\omega_{\PP(\cE_{x})} \otimes \Theta_{\rM(r, L(-x))}^{r+1}) &\hookrightarrow \rH^{0}(\rM(r, L(-x)), \omega_{\rM(r, L(-x))} \otimes \Theta_{\rM(r, L(-x))}^{r+1})\\
	&= \rH^{0}(\rM(r, L(-x)), \Theta_{\rM(r, L(-x))}^{-r+1}) = 0.
\end{split}
\]
\end{proof}

When $g = 2$, the only cohomologies that are not covered by the above vanishing results are $i = r-1$ and $0 \le j \le r-3$. Thus, the above proof provides the following statement. 

\begin{corollary}\label{cor:ACMthetak}
If $g \ge 2$, $\cE_{x} \otimes \Theta^{-1}$ is ACM with respect to $\Theta^{k}$ for $k \ge r-1$.
\end{corollary}

\begin{remark}\label{rmk:g=2}
\begin{enumerate}
\item The vanishing result in \text{Step 3} is proved in \cite[Proposition 19]{BM19} with a different method, under the assumption of $g \ge 3$. Our approach is valid for $g = 2$ as well.
\item When $g = r = 2$, $\rM(r, L)$ is an intersection of two quadrics and $\cE_{x}$ is a spinor bundle (\cite{CKL19, FK18}). From this description, it was shown that $\cE_{x}$ is ACM for all $x \in X$.
\end{enumerate}
\end{remark}

\begin{question}\label{que:g=2}
Can we extend Theorem \ref{thm:ACMintro} to the $g = 2$ case?
\end{question}

\begin{remark}\label{rmk:Ulrich}
The bundle $\cE_{x}$ is not Ulrich in general. For instance, if $g = r = 2$, $h^0(\rM(r, L),\cE_x)= 4 < 8 = 2 \deg(\rM(r, L))$. It is an interesting problem to construct Ulrich bundles on $\rM(r,L)$. See \cite{CKL19} for an alternative construction of Ulrich bundles for $g=r=2$ case.
\end{remark}


\bibliographystyle{alpha}

\begin{thebibliography}{GOST15}

\bibitem[BGM18]{BGM18} P. Belmans, S. Galkin and S. Mukhopadhyay. Semiorthogonal decompositions for moduli of sheaves on curves. {\em Oberwolfach Report} No. 24/2018, 9--11.

\bibitem[BM19]{BM19}
{P. Belmans and S. Mukhopadhyay}, Admissible subcategories in derived categories of moduli of vector bundles on curves. {\em Adv. Math.} 351 (2019), 653--675. 

\bibitem[BO95]{BO95}
{A. Bondal and D. Orlov}, Semiorthogonal decomposition for algebraic varieties. Preprint, arXiv:alg-geom/9506012, 1995.

\bibitem[CKL19]{CKL19}
{Y. Cho, Y. Kim and K.-S. Lee}, Ulrich bundles on intersection of two 4-dimensional quadrics. International Mathematics Research Notices, 2019;, rnz320, https://doi.org/10.1093/imrn/rnz320

\bibitem[DN89]{DN89}
{J.-M. Drezet and M. Narasimhan}, Groupe de Picard des vari\'et\'es de modules de fibr\'es semi-stables sur les courbes alg\'ebriques. {\em Invent. Math.} 97 (1989), no. 1, 53--94.

\bibitem[Eis80]{Eis80}
{D. Eisenbud}, Homological algebra on a complete intersection, with an application to group representations. {\em Trans. Amer. Math. Soc.} 260, 35--64 (1980)

\bibitem[ES03]{ES03}
{D. Eisenbud and F.-O. Schreyer}, Resultants and Chow forms via exterior syzygies. With an appendix by Jerzy Weyman. {\em J. Amer. Math. Soc.} 16 (2003), no. 3, 537--579. 

\bibitem[EP04]{EP04}
{E. Esteves and M. Popa}, Effective very ampleness for generalized theta divisors. {\em Duke Math. J.} 123 (2004), no. 3, 429--444. 

\bibitem[FK18]{FK18}
{A. Fonarev and A. Kuznetsov}, Derived categories of curves as components of Fano manifolds. {\em J. London Math. Soc.} (2) 97 (2018) 24--46.

\bibitem[GL20]{GL20}
{T. Gomez and K.-S. Lee}, Motivic decompositions of moduli spaces of vector bundles on curves. Preprint, arXiv:2007.06067.  

\bibitem[Gro05]{Gro05}
{A. Grothendieck}, {\em Cohomologie locale des faisceaux coh\'erents et th\'eor\`emes de Lefschetz locaux et globaux (SGA 2)}. S\'eminaire de G\'eom\'etrie Alg\'ebrique du Bois Marie, 1962. Augment\'e d'un expos\'e de Mich\`ele Raynaud. With a preface and edited by Yves Laszlo. Revised reprint of the 1968 French original. 4. Soci\'et\'e Math\'ematique de France, Paris, 2005. x+208 pp.

\bibitem[GOST15]{GOST15}
{Y. Gongyo, S. Okawa, A. Sannai, and S. Takagi.} Characterization of
varieties of Fano type via singularities of Cox rings. {\em J. Algebraic Geom.}, 24(1):159--182, 2015.

\bibitem[Har67]{Har67}
{R. Hartshorne}, {\em Local cohomology.} A seminar given by A. Grothendieck, Harvard University, Fall, 1961. Lecture Notes in Mathematics, No. 41 Springer-Verlag, Berlin-New York 1967 vi+106 pp.

\bibitem[Har77]{Har77}
{R. Hartshorne}, {\em Algebraic geometry}. Graduate Texts in Mathematics, No. 52. Springer-Verlag, New York-Heidelberg, 1977. xvi+496 pp.

\bibitem[HL10]{HL10}
{D. Huybrechts and M. Lehn}, {\em The geometry of moduli spaces of sheaves. Second edition.} Cambridge Mathematical Library. Cambridge University Press, Cambridge, 2010. xviii+325 pp.

\bibitem[Kol86]{Kol86}
{J. Koll\'ar}, 
Higher direct images of dualizing sheaves. I. {\em Ann. of Math.} (2) 123 (1986), no. 1, 11--42. 

\bibitem[LN05]{LN05}
{H. Lange and P. E. Newstead}, On Poincar\'e bundles of vector bundles on curves. {\em Manuscripta Math.} 117 (2005), no. 2, 173--181. 

\bibitem[Laz04]{Laz04}
{R. Lazarsfeld}, {\em Positivity in algebraic geometry. II. Positivity for vector bundles, and multiplier ideals.} Ergebnisse der Mathematik und ihrer Grenzgebiete. 3. Folge. A Series of Modern Surveys in Mathematics, 49. Springer-Verlag, Berlin, 2004. xviii+385 pp. 

\bibitem[Lee18]{Lee18}
{K.-S. Lee}, Remarks on motives of moduli spaces of rank 2 vector bundles on curves. Preprint, arXiv:1806.11101.

\bibitem[LN21]{LN21}
{K.-S. Lee and M. S. Narasimhan}, Symmetric products and moduli spaces of vector bundles of curves. Preprint, to appear.

\bibitem[MS80]{MS80}
{V. B. Mehta and C. S. Seshadri}, 
Moduli of vector bundles on curves with parabolic structures. {\em Math. Ann.} 248, 205--239 (1980).

\bibitem[MY20]{MY20}
{H-B. Moon and S-B. Yoo},
Finite generation of the algebra of type A conformal blocks via birational geometry II: higher genus. {\em Proc. Lond. Math. Soc.} (3) 120 (2020), no. 2, 242--264.

\bibitem[MY21]{MY21}
{H-B. Moon and S-B. Yoo},
Finite generation of the algebra of type A conformal blocks via birational geometry. {\em Int. Math. Res. Not. IMRN}, (2021), no. 7, 4941--4974.

\bibitem[MFK94]{MFK94}
{D. Mumford, J. Fogarty, and F. Kirwan}, {\em Geometric invariant theory. Third edition.}, Ergebnisse der Mathematik und ihrer Grenzgebiete (2), 34. Springer-Verlag, Berlin, 1994. xiv+292 pp.

\bibitem[Nar17]{Nar17}
{M. S. Narasimhan}, Derived categories of moduli spaces of vector bundles on curves. {\em J. Geom. Phys.} 122 (2017), 53--58. 

\bibitem[Nar18]{Nar18}
{M. S. Narasimhan}, Derived categories of moduli spaces of vector bundles on curves II. {\em Geometry, algebra, number theory, and their information technology applications}, 375-382, Springer Proc. Math. Stat., 251, Springer, Cham, 2018.

\bibitem[NR75]{NR75}
{M. S. Narasimhan and S. Ramanan}, Deformations of the Moduli Space of Vector Bundles Over an Algebraic Curve. {\em Annals of Mathematics}, vol. 101, no. 3, (1975), 391–417.

\bibitem[New78]{New78}
{P. E. Newstead}, {\em Introduction to Moduli Problems and Orbit Spaces}, volume 51 of TIFR Lectures on Mathematics and Physics. Bombay: Tata Institute of Fundamental Research, 1978.

\bibitem[Ram73]{Ram73}
{S. Ramanan}, The moduli spaces of vector bundles over an algebraic curve. {\em Math. Ann.} 200 (1973), 69--84.

\bibitem[Sun00]{Sun00}
{X. Sun}, Degeneration of moduli spaces and generalized theta functions. {\em J. Algebraic Geom.} 9 (2000), no. 3, 459--527. 

\bibitem[Yok95]{Yok95}
{K. Yokogawa}, Infinitesimal deformation of parabolic Higgs sheaves. {\em Internat. J. Math.} 6 (1995), no. 1, 125--148.

\bibitem[Yos90]{Yos90}
{Y. Yoshino}, {\em Cohen-Macaulay modules over Cohen-Macaulay rings.} 
London Mathematical Society Lecture Note Series, 146. Cambridge University Press, Cambridge, 1990. viii+177 pp.

\end{thebibliography}

\end{document}